\documentclass[11pt]{article} \topmargin -0.25cm
\usepackage{amsmath}
\usepackage{amssymb}

\newcommand{\vf}{\varphi}
\newcommand{\R}{\mathbb R}
\newcommand{\ccc}{{\cal C}^1(\Rd,\Rd\otimes\Rd)}
\newcommand{\cccc}{{\cal C}^1(\Rd,\Rd)}
\newcommand{\Rp}{\mathbb R^+}
\newcommand{\Rd}{\mathbb R^d}
\newcommand{\D}{\mathbb D\,(\Rp,\Rd)}
\newcommand{\Ddd}{\mathbb D\,(\Rp,\R^{2d})}
\newcommand{\Dddd}{\mathbb D\,(\Rp,\R^{3d})}
\newcommand{\Ds}{\mathbb D\,(\Rp,\R^}

\newcommand{\Dj}{\mathbb D\,(\Rp,\R)}
\newcommand{\Dii}{\mathbb D\,(\Rp,\R^{2d})}

\newcommand{\ro}{r_0}
\newcommand{\lra}{\longrightarrow}
\newcommand{\ov}{\overline}

\newcommand{\wh}{\widehat}

\newcommand{\wy}{\widehat Y}

\newcommand{\wk}{\widehat K}
\newcommand{\wz}{\widehat Z}
\newcommand{\wx}{\widehat X}
\newcommand{\wj}{\widehat J}

\newcommand{\rn}{^{\rho^n}}
\newcommand{\rnn}{^{n,\rho^n}}
\newcommand{\tk}{t^n_{k}}
\newcommand{\tkn}{t^n_{k+1}}
\newcommand{\tkk}{t^n_{k-1}}
\newcommand{\N}{\mathbb N}

\newcommand{\normalshape}{\rm}

\newcommand{\arrowd}{\mathop{\longrightarrow}_{\cal D}}

\newcommand{\arrowp}{\mathop{\longrightarrow}_{\cal P}}
\newtheorem{theorem}{\bf Theorem}[subsection]
\newtheorem{lemma}[theorem]{\bf Lemma}
\newtheorem{corollary}[theorem]{\bf Corollary}
\newtheorem{remark}[theorem]{\bf Remark}
\newenvironment{proof}{{\sc Proof.}}{\hfill $\Box$}

\newcommand{\nsubsection}{\setcounter{equation}{0}\subsection}

\begin{document}
\title{\bf On reflected  Stratonovich  stochastic differential equations}
\author{Leszek S\l omi\'nski\footnote{E-mail address: leszeks@mat.umk.pl;
tel.: +48-566112954; fax: +48-566112987.}\\
 \small Faculty of Mathematics and Computer Science,
Nicolaus Copernicus University,\\
\small ul. Chopina 12/18, 87-100 Toru\'n, Poland}
\date{}
\maketitle
\begin{abstract}
We study the problem of existence, uniqueness and  approximation
of  solutions of finite dimensional Stra\-ton\-ovich stochastic
differential equations with reflecting boundary condition driven
by  semimartingales with jumps. As an application we generalize
known results on the Wong-Zakai type approximations.
\end{abstract}

\nsubsection{ Introduction}

Let  $D$  be a connected domain in $\Rd$. We consider a
$d$-dimensional stochastic differential equations (SDEs) on a
domain $D$ with reflecting boundary condition of the form
\begin{equation}
\label{eq1.1} X_t=X_0+\int_0^tf(X_s)\circ dZ_s+K_t,\quad t\in\Rp.
\end{equation}
Here the notation  `$ \circ$'  indicates that we deal with
stochastic integral of Stratonovich type. In (\ref{eq1.1}) $X_0\in
\bar D=D\cup\partial D$, $X$ is a reflecting process on $\bar D$,
$K$ is a bounded variation process with  variation $|K|$
increasing only when $X_t\in\partial D$  and  $Z$ is a
$d$-dimensional semimartingale with jumps.

In the nonreflected case,  that is, when $D=\R^d$ and $K=0$ the
above type of SDEs has been investigated in many papers (see,
e.g., \cite{KP1,KPP,Ma,Ma1,Mar, Sl,Sl3,WZ,WZ1}). In  particular, Wong
and Zakai \cite{WZ,WZ1} proved that for $Z$ being a Brownian
motion the solution of (\ref{eq1.1}) is a  limit  of the following
simple approximation scheme: $\wh X^n_0=X_0$,
\[
\wh X^n_t=\wh X^n_{\frac{k}{n}}+n\int_{\frac{k}{n}}^tf(\wh
X_s^n)\,ds(Z_{\frac{k+1}{n}}-Z_{\frac{k}{n}}), \quad
t\in(\frac{k}{n},\frac{k+1}{n}],\,k\in\N\cup\{0\}.
\]
Since that time  approximations of the above type are called
Wong-Zakai approximations. The case of nonreflected Stratonovich
SDEs driven by semimartingales with jumps has been  studied by
Mackevicius \cite{Ma}, Marcus \cite{Ma1,Mar} and in the general
case by Kurtz, Pardoux and Protter  \cite{KPP}. Stratonovich SDEs
considered in the last paper has  the following form:
\begin{align*}
X_t=X_0+\int_0^tf(X_{s})\circ dZ_s
=X_0&+\int_0^tf(X_{s-})\,dZ_s+\frac12\int_0^tf^{\prime}f
(X_{s-})\,d[Z]^c_s\\&+\sum_{s\leq t} \{\vf(f\Delta
Z_s,X_{s-})-X_{s-}-f(X_{s-})\Delta Z_s\},
\end{align*}
where  for given $g\in\cccc$, $\vf(g,x)$  denotes the value at
time  $u=1$  of the solution of the ordinary differential equation
$ {\frac{dy}{du}}(u)=g(y(u))$, $y(0)=x\in\Rd$. In particular,
Kurtz, Pardoux and Protter were able to prove that $\wh
X^{n}_t\arrowp X_t$ except possibly  for a countable set of $t'$s.
This convergence cannot be in general strengthened   to the
convergence in the  Skorokhod topology $J_1$  (in contrast to the
approximating sequence  $\{\wh X^n\}$, the solution $X$ need not
have continuous trajectories).

Reflected Stratonovich SDEs has been studied for the first time in
Doss and Priouret \cite{DP}. In \cite{DP} the convergence of
Wong--Zakai type approximations has been proved in the case of
diffusion processes and sufficiently smooth $\partial D$. Next,
the almost sure convergence of Wong--Zakai approximation has been
observed observed by Pettersson \cite{Pe} in the case of  convex
domain and constant diffusion coefficient. Pettersson's results
have been refined by Ren and Xu \cite{RX,RX1} in papers devoted to
multivalued SDEs, which correspond to SDEs with reflecting
boundary conditions on convex domains. Quite recently, Evans and
Stroock \cite{ES} have proved  weak convergence of Wong--Zakai
approximations in the case of not necessary smooth domains
satisfying  conditions (A), (B)  and (C) from the paper by Lions
and Sznitman \cite{LSz}. Their result has been strengthened to
$L^p$ convergence by Aida and Sasaki \cite{AS}. The  Wong-Zakai
approximation of reflected diffusion processes has also been  studied by
Zhang \cite{zh1,zh2}. We emphasize that in all the above mentioned
papers the limit of Wong--Zakai approximations is a continuous
solution to some reflected Stratonovich SDE.

In the present paper we remove  condition (C) from \cite{LSz} and
we only assume that $D$  satisfies conditions (A), (B).  We study the general reflected Stra\-ton\-ovich
SDEs driven by semimartingales with jumps (\ref{eq1.1}) and
various methods of its  approximations including Wong--Zakai type
approximations. Since the  Stratonovich stochastic integral
considered in this paper coincides with the general Stratonovich
stochastic integral studied in the paper by Kurtz, Pardoux and
Protter, our results generalize appropriate results from the above
mentioned papers. The problem of  existence and uniqueness of
solutions to reflected Stratonovich SDEs with jumps in domains
satisfying conditions (A), (B) has been also considered  by
Kohatsu-Higa \cite{KH} but the Stratonovich integral in \cite{KH}
is different from our integral. In \cite{KH} the integral  has the
property that its jumps place the solution inside the domain. This
implies that the compensating reflection process $K$ has
continuous trajectories and its variation is increasing only when the solution
$X$ is living on the boundary in a continuous manner. Using very
subtle and difficult methods based on the  multidimensional change
of time (see, e.g., \cite{Kur}) Kohatsu-Higa was able to prove the
existence and uniqueness results for such reflected SDEs, provided
that  the driving semimartingale $Z$ has summable jumps, i.e.
$\sum_{s\leq t}|\Delta Z_s|<\infty$, $P$-a.s., $t\in\Rp$.

The paper is organized as follows.

Section 2 presents some preliminaries concerning solutions  of the
Skorokhod problem and solutions of reflected SDEs.

In  Section 3 we study stability of reflected SDEs of the  form
(\ref{eq1.1}). More precisely, we consider a sequence of
semimartingales $\{Z^n\}$ satisfying the condition (UT) and a
sequence  $\{(X^n,K^n)\}$ of solutions of SDEs of the form
(\ref{eq1.1}), i.e.
\[
X^n_t=X^n_0+\int_0^tf(X^n_s)\circ dZ^n_s+K^n_t,\quad t\in\Rp.
\]
We  give conditions ensuring weak and strong  convergence  of
$\{(X^n,K^n)\}$  to  the solution $(X,K)$  of (\ref{eq1.1}).
Consequently, we get the existence of weak solution  of
(\ref{eq1.1}) provided that $f,f'f$ are continuous and bounded and
$|\Delta Z|<\ro/\sup_{x\in \bar D}||f(x)||$, where $\ro$ is some
constant depending on the domain $D$.  If additionally $f,f'f$ are
locally Lipschitz continuous, we prove the existence and
uniqueness of strong solutions to (\ref{eq1.1}).

Section 4 is devoted to the study of Wong-Zakai type
approximations of the unique strong solution $(X,K)$ of
(\ref{eq1.1}). We consider two approximation schemes. The first
one is defined by the recurrent formula: $\wh X^n_0=X_0$ and
\[
\wh X^n_t=\Pi_{\bar D}(\wh X^n_{\frac{k}{n}-})+n\int_{\frac{k}{n}}^tf(\wh
X_s^n)\,ds(Z_{\frac{k+1}{n}}-Z_{\frac{k}{n}}), \quad
t\in[\frac{k}{n},\frac{k+1}{n}),\,k\in\N\cup\{0\},
\]
where $\Pi_{\bar D}(x)$ is the projection of $x$ on $\bar D$.
Applying the approximation results from  Section 3 we show that
$\wh X^{n}\arrowp X$ in the $S$ topology introduced by Jakubowski
\cite{Jak},  and that $ \wh X^{n}_t\arrowp X_t$ provided  that
$\Delta Z_t=0$, $t\in\Rp$. The $S$ topology is weaker than the
Skorokhod topology $J_1$ but  stronger than the Meyer-Zheng
topology (see, e.g., \cite{Kur,MZ}).  In the general case our
convergence results cannot be  strengthened   to the convergence
in the Skorokhod topology $J_1$. However, in case  $Z$ has
continuous trajectories we prove that $\sup_{t\leq q}|\wh
X^n_t-X_t|\arrowp0$,  $q\in\Rp$. The second scheme has the
following form: $\bar X^n_0=X_0$ and
\[
\bar
X^{n}_t=\bar X_{\frac{k}{n}}+n\int_{\frac{k}{n}}^tf(\bar
X^{n}_{s})\,ds(Z_{\frac{k+1}{n}}-Z_{\frac{k}{n}})+\bar K^n_t-\bar
K^n_{\frac{k}{n}},\quad
t\in[\frac{k}{n},\frac{k+1}{n}),\,k\in\N\cup\{0\},
\]
where in each step some appropriate deterministic reflected
differential equation  is solved. This is well known method of
approximation of reflected diffusions (see, e.g.,
\cite{AS,DP,ES,Pe,RX,RX1}).  We show that  for any continuous
semimartingale $Z$,  $\sup_{t\leq q}|\bar X^n_t-X_t|\arrowp0$,
$q\in\Rp$. Unfortunately, this method is not applicable  in case
of equations of the form (\ref{eq1.1}) driven by semimartingales
with  jumps (see Remark \ref{rem4}).

We will use the  following  notation. $\Rp=[0,\infty)$. $\D$ is
the space of all c\`adl\`ag  mappings $x:\Rp\to\R^d$, i.e.
mappings which are right continuous and admit left-hands limits
equipped with the  Skorokhod $J_1$ topology. For $x\in\D$, $t>0$,
we write $x_{t-}=\lim_{s\uparrow t}x_s$, $\Delta x_t=x_t-x_{t-}$.
For $x$ with locally bounded variation we denote by $|x|_t$ its
total variation  on the interval $[0,t]$, i.e., $|x|_t = \sup_\pi
\sum_{j=1}^n |x_{t_j}-x_{t_{j-1}}| <+\infty$, where the supremum
is taken over all subdivisions $\pi=\{0=t_0<\ldots<t_n=t\}$ of
$[0,t]$  and $|\cdot|$ denotes the usual Euclidean norm in $\Rd$.
Every process $Z$ appearing in the sequel is assumed to have
c\`adl\`ag trajectories. If $Z=(Z^1,\dots,\,Z^d)$ is a
semimartingale then $[Z]_t$ stands for $\sum^d_{i=1}[Z^i]_t$ and
$[Z^i]$ stands for the quadratic variation process of $Z^i$,
$i=1,\dots,\,d$. Similarly, $\langle Z\rangle_t
=\sum^d_{i=1}\langle Z^i\rangle_t$ and $\langle Z^i\rangle$ stands
for  the predictable compensator of $[Z^i]$, $i=1,\dots,\,d$.
"$\arrowd$", "$\arrowp$" denote convergence in law and in
probability, respectively.

\nsubsection{Preliminaries}

Let  $D$  be a nonempty connected domain in $\Rd$. We define the
set ${\cal N}_x$  of inward normal unit vectors at  $x\in\partial
D$ by
\[
{\cal
N}_x=\bigcup_{r>0}{\cal N}_{x,r}\,,\,\,\,\,{\cal N}_{x,r}=\{\mbox{\bf
n}\in\Rd;\,|\mbox{\bf n}|=1\,,B(x-r\mbox{\bf n},\,r)\cap D=\emptyset
\},
\]
where $ B(z,r)=\{y\in\Rd;|y-z|<r\}$, $z\in\Rd$, $r>0$. Following
Lions and Sznitman \cite{LSz} and Saisho \cite{Sa} we  consider
two assumptions.
\begin{description}
\item[{\rm (A)}] There exists a constant
$\ro>0$  such that
\[
{\cal N}_x={\cal N}_{x,\ro}\neq\emptyset
\quad\mbox{\rm for every}\quad
x\in\partial D.
\]
\item[{\rm (B)}]
There exist constants  $\delta >0\,,\,\beta\geq 1$ such
that for every  $x\in\partial D$ there exists a unit vector $\mbox{\bf
l}_x$  with the following property
\[
<\mbox{\bf l}_x\,,\,\mbox{\bf n}>\,\geq\frac1{\beta}\quad\mbox{\rm for
every}\quad \mbox{\bf n}\in\bigcup_{y\in B(x,\delta)\cap\partial D}{\cal N}_y
\]
where $<\cdot,\cdot>$  denotes the usual inner product in $\Rd$.
\end{description}
\begin{remark}{\rm (\cite{LSz,Sa,Ta})
\label{rem2.1}
\begin{description}
\item[{\rm (i)}] {\bf n}$\in{\cal N}_{x,r} $ if and only if
$<y-x,\mbox{\bf n}>+\frac1{2r}|y-x|^2\geq 0 $  for every $y\in\bar
D$.
\item[{\rm (ii)}]
If $\mbox{\rm dist}(x,\bar D)\,<\,\ro$, $x\notin\bar D$  then
there exists a unique  $\Pi_{\bar D}(x)\in\bar D$ such that
$|x-\Pi_{\bar D}(x)|=\mbox{\rm dist}(x,\bar D)$. Moreover,
$(\Pi_{\bar D}(x)-x)/|\Pi_{\bar D}-x|\in{\cal N}_{ \Pi_{\bar
D}(x)}$
\item[{\rm (iii)}]
If  $D$  is a convex domain in  $\Rd$  then  $\ro=+\infty$.
\end{description}}
\end{remark}

Let $y\in\D$ and $y_0\in\bar D$. We recall that a
pair $(x,k)$ $\in\Dii $  is a solution of the Skorokhod problem
associated with  $y$ if
\begin{description}
\item[{\rm -}] $ x_t=y_t+k_t\,,\,t\in\Rp$,
\item[{\rm -}]$ x_t\in\bar D\,,\,t\in\Rp$,
\item[{\rm -}] $k$ is a function with bounded variation on each finite
interval such that $k_0=0$  and
\[
k_t=\int_0^t\mbox{\bf n}_s\,d|k|_s\,,\quad |k|_t=\int_0^t\mbox{\bf
1}_{\{x_s\in\partial D\}}\,d|k|_s\,,
\]
where $ \mbox{\bf n}_s\in{\cal N}_{x_s}$  if  $x_s\in\partial D$.
\end{description}

The problem of existence and approximation of solutions of the
Skorokhod problem in domains satisfying  conditions (A) and (B)
has been studied  by Saisho \cite{Sa} (the case of continuous
functions) and S\l omi\'nski \cite{Sl1} (the case of c\`adl\`ag
functions). We now recall the general approximation method
considered in these papers.

Let $\{\{\tk\}\}$ be an array of nonnegative numbers such that in
the $n^{\mbox{\tiny th}}$ row  the sequence $T_n=\{\tk\}$ forms a
partition of $\Rp$ such that $0=t^n_{0}<t^n_{1}<\dots$,
$\lim_{k\rightarrow\infty}\tk=+\infty$ and
\begin{equation}
\label{eq2.1} \max_k\,(\tk-\tkk)\longrightarrow0\quad
\mbox{\normalshape as}\,\,n\rightarrow+\infty.
\end{equation}
Given $\{\{\tk\}\}$  we define a sequence of summation rules
$\{\rho^n\},\,\rho^n:\Rp\lra\Rp$ by $\rho^n_t=\max\{\tk;\tk\leq
t\}$. For every $y\in\D$  the sequence $\{y\rn\}$ denotes the
following  discretizations of $y$:
\[
y\rn_t=y_{\rho^n_t}=y_{\tk}\quad\mbox{\normalshape  for}\quad
t\in[\tk,\tkn),\,k\in\N\cup\{0\},\,n\in\N.
\]
Using for instance \cite[Proposition 3.6.5]{EK} one can check that
$y\rn\to y$ in $\D$. From \cite[Corollary 3]{Sl1} it follows that
if (A)  and (B) are satisfied then for every $y\in\D$  such that
$y_0\in\bar D$ and $|\Delta y|<\ro$ there exists a unique solution
$(x,k)$ of the Skorokhod problem associated with $y$. Moreover, if
$\{(x^n,k^n)\}$ is the sequence of solutions of the Skorokhod
problem associated with the sequence of discretizations $\{y\rn\}$
then
\begin{equation}
\label{eq2.2}(x^n,k^n,y\rn)\lra(x,k,y)
\quad\mbox{\rm in}\,\,\Dddd.
\end{equation}
It is worth noting that for all sufficiently large $n$,
$(x^n,k^n)$ are defined by the following recurrent formula:
$x^n_0=y_0$, $k^n_0=0$,
\[
\left\{\begin{array}{ll}
x^n_{\tkn}&=\Pi_{\bar D}\big(x^n_{\tk}+(y_{\tkn}-y_{\tk})\big),\\
k^n_{\tkn}&=k^n_{\tk}+(x^n_{\tkn}-x^n_{\tk})-(y_{\tkn}-y_{\tk})
\end{array}
\right.
\]
and $x^n_t=x^n_{\tk}$, $k^n_t=k^n_{\tk}$,  $t\in[\tk,\tkn)$,
$k\in\N\cup\{0\}$.
\begin{lemma}
\label{lem1} Assume (A) and (B). Let $y\in\D$ be such that
$y_0\in\bar D$, $|\Delta y|<\ro$,  and let $(x,k)$ denote the
solution of the  Skorokhod problem associated with $y$. Then for
every $0\leq t<q$
\[
|k|_{[t,q]}\leq|y|_{[t,q]}\quad and \quad |x|_{[t,q]}\leq 2|y|_{[t,q]}.
\]
\end{lemma}
\begin{proof}
Without loss of generality we may assume that $t=0$.  Let
$\{(x^n,k^n)\}$ be a sequence of solutions of the Skorokhod
problem associated with the sequence of discretizations
$\{y\rn\}$. Clearly,
\[
|\Delta k^n_{\tkn}|=|k^n_{\tkn}-k^n_{\tk}|
\leq |y_{\tkn}-y_{\tk}|=|\Delta y\rn_{\tkn}|,\quad k\in\N\cup\{0\},
\]
which implies that for every $q\in\Rp$,
\[
|k^n|_q\leq\sum_{k;\tkn\leq q}|k^n_{\tkn}-k^n_{\tk}|
\leq \sum_{k;\tkn\leq q}|y_{\tkn}-y_{\tk}|=| y\rn|_{q}|.
\]
Since $k^n\to k$ in $\D$,
\[
|k|_q\leq\liminf_{n\to\infty}|k^n|_q
\leq\sup_n|y\rn|_q\leq|y|_q,\quad \mbox{\rm provided
that}\,\,\Delta y_q=0.
\]
If $\Delta y_q\neq0$ then there exists a sequence~$\{q_i\}$ such
that $q_i\downarrow q$ and $\Delta y_{q_i}=0$, $i\in\N$. Then
$|k|_{q_i}\leq|y|_{q_i}$, $i\in\N$,  so letting $i\to\infty$ we
obtain the desired result.
\end{proof}
\medskip

Lemma \ref{lem1} is a  generalization  of
\cite[Theorem 2.1]{LSz}, where the case of continuous $y$ and smooth domains is considered. In \cite{AS} estimates of Lemma \ref{lem1} but with greater constants were proved for continuous $y$ and domains satisfying condition (A) only.

Let  $(\Omega\,,\,\cal F\,,\,\cal P)$  be a probability space and
let $({\cal F}_t)$  be a filtration on  $(\Omega\,,\,\cal
F\,,\,\cal P)$ satisfying the usual conditions. Let  $Y$  be an
$({\cal F}_t)$  adapted process and  $Y_0\in\bar D$. We say that a
pair  $(X,K)$  of  $({\cal F}_t)$ adapted processes solves the
Skorokhod problem associated with   $Y$  if and only if  for every
$\omega\in\Omega\,,\,(X(\omega),K(\omega))$ is a solution of the
Skorokhod problem corresponding to  $Y(\omega)$. Let us note that
by \cite[Corollary 6.10]{Sl1} for every process $Y$ such that
$Y_0\in\bar D$ and $|\Delta Y|<\ro$ there exists a unique solution
of the Skorokhod problem associated with $Y$.

We will consider processes  $Y, Y'$  admitting the decompositions
\begin{equation}
\label{eq2.3} Y_t=H_t+M_t+V_t,\quad  Y'_t=H_t+  M'_t+ V'_t, \quad
t\in\Rp,
\end{equation}
 where $H$ is an $({\cal F}_t)$ adapted  process,
$M,M'$  are $({\cal F}_t)$ adapted local martingale with
$M_0=M'_0=0$ and $V, V'$ are $({\cal F}_t)$ adapted processes of
bounded variation with $V_0= V'_0=0$.

\begin{remark}{\rm(\cite{Sl1,Sl2})
\label{rem2} Assume that $Y, Y'$ admit decompositions
(\ref{eq2.3}) and $Y_0, Y'_0\in\bar D$. Let $(X,K)$, $(X',K')$
denote solutions of the  Skorokhod problems associated with
processes $Y$ and $Y'$, respectively. If $|\Delta Y|,|\Delta Y'|$
$\leq\frac{\ro}4 $  and there exists a constant $a$  such that
$|K|_{\infty},| K'|_{\infty}\leq a$ (in the case $r_0<\infty$)
then for every $p\in\N$ there exists a constant $ C(p)$ depending
on  $p$  (and also on  $a,\,\ro\,,\,\beta\,,\,\delta\,$ ) such
that
\begin{equation}
\label{eq2.4}E\sup_{t \leq\tau}|X_t- X'_t|^{2p} \leq C(p)
E([M-M']_{\tau }^p + |V-V'|_{\tau }^{2p})
\end{equation}
and
\begin{equation}\label{eq2.5}
 E\sup_{t < \tau}|X_t- X'_t|^{2p} \leq C(p)
E( [M-M']_{\tau -}^p + |V- V'|_{\tau -}^{2p} + \langle M- M'
\rangle ^p_{\tau -})
\end{equation}
for every stopping time $\tau$. }
\end{remark}

Let $X_0\in\bar D$ and let $Z_t=(Z^1_t,\dots,Z^d_t)\,$  be an
$({\cal F}_t)$  adapted semimartingale such that $Z_0=0$. Given
$f:\ov D\lra\Rd\otimes\Rd$, $f(x)=\{f_{ik}\}_{i,k=1,\dots,d}$ such
that $f$  belongs to $\ccc$ we consider  reflected Stratonovich
SDE  of the form
\begin{align}
X_t&=X_0+\int_0^tf(X_{s})\circ dZ_s+K_t\label{eq2.6}\\
&=X_0+\int_0^tf(X_{s-})\,dZ_s+\frac12\int_0^t
f^{\prime}f(X_{s-})\,d[Z]^c_s\nonumber\\
&\quad+\sum_{s\leq t} \{\vf(f\Delta
Z_s,X_{s-})-X_{s-}-f(X_{s-})\Delta Z_s\}+K_t\nonumber,
\end{align}
where  for given  $g\in\cccc$,  $\vf(g,x)$  denotes the value
at time  $u=1$  of the solution of
the following ordinary differential equation:
\[
{\frac{dy}{du}}(u)=g(y(u));\quad y(0)=x\in\Rd.
\]
Equivalently, for  $i=1,\dots,d$,
\begin{align*}
X^i_t&=X^i_0+\sum_{j=1}^d\int_0^tf_{ij}(X_{s-})\,dZ^j_s
+\frac12\sum_{j=1}^d\sum_{l=1}^d\sum_{m=1}^d
\int_0^t\frac{\partial f_{ij}}{\partial x_l}f_{lm}(X_{s-})\,d[Z^j,Z^m]^c_s\\
&\quad+\sum_{s\leq t}\{\vf^i(f\Delta
Z_s,X_{s-})-X^i_{s-}-\sum_{j=1}^d f_{ij}(X_{s-})\Delta
Z^j_s\}+K^i_t,\quad t\in\Rp.
\end{align*}
Note that in case $Z$ has continuous trajectories the definition
of Stratonovich stochastic integral given above  coincides with
the well known definition given by Meyer \cite{Me1}.

We say that the SDE (\ref{eq2.5})  has a strong solution if there
exists a pair $(X,K)$  of  $({\cal F}_t)$ adapted processes which
solves the Skorokhod problem associated with
\[
Y_t=X_0+\int_0^tf(X_{s})\circ dZ_s,\quad t\in\Rp.
\]

In the sequel we say that $f,f^{\prime}f$ have some property if
the coefficients $f_{ij}$, $\frac{\partial f_{ij}}{\partial
x_l}f_{lm}$  have this property for
 $i,j,l,m$ $=1,\dots,d$.

\begin{remark}\label{rem3}
{\rm (a) The existence and uniqueness  of strong solutions of
usual SDEs driven by semimartingales  with reflecting  boundary on
domains satisfying conditions (A) and (B) has been studied in
detail in \cite{Sl1}.  In particular, in \cite[Theorem 5]{Sl1} it
is proved that if $f$  is bounded and Lipschitz continuous,
$X_0\in\bar D$ and $|\Delta Z|<\ro/L$, where $L=\sup_{x\in\bar
D}\|f(x)\|$, then there exists a unique strong solution of the
reflecting SDE of the form
\[X_t=X_0+\int_0^t f(X_{s-})\,dZ_s +K_t,\quad t\in\Rp.\]
Since (\ref{eq2.6}) has the additional term $\sum_{s\leq t}
\{\vf(f\Delta Z_s,X_{s-})-X_{s-}-f(X_{s-})\Delta Z_s\}$,
$t\in\Rp$,  \cite[Theorem 5]{Sl1} does not apply to  Stratonovich
SDE with reflection.

(b) If $f'f$ is bounded then be results proved in \cite[page 356]{KPP} there
is  $C>0$ such that for every $s\Rp$ we have $|\vf(f\Delta
Z_s,X_{s-})-X_{s-}-f(X_{s-})\Delta Z_s|\leq C|\Delta Z_s|^2$.
Hence
\[
\sum_{s\leq t}
|\vf(f\Delta Z_s,X_{s-})-X_{s-}-f(X_{s-})\Delta Z_s|\leq
C\sum_{s\leq t}|\Delta Z_s|^2, \quad t\in\Rp,
\]
which means that the additional term is finite a.s. One can
observe that as in the nonreflected case  considered in
\cite{KPP}, $(X,K)$ is a strong solution of (\ref{eq2.6}) iff it
satisfies
\begin{align}
X_t&=X_0 +\int_0^tf(X_{s-})\,dZ_s+\frac12\int_0^t
f^{\prime}f(X_{s-})\,d[Z]^c_s
\label{eq2.7}\\
& \,\,+\int_0^th(s,\cdot,X_{s-})d[Z]^d_s +K_t,\quad
t\in\Rp,\nonumber
\end{align}
where $h(s,\omega,x)=(\vf(f\Delta Z_s,x)-x-f(x)\Delta Z_s)/|\Delta
Z_s|^2$, $s\in\Rp$, $x\in\bar D$ is bounded by the constant $C$.
Moreover, if $f,f'f$ are Lipschitz continuous and the jumps of $Z$
are bounded, then there is another constant $C'$ such that
\begin{equation}\label{eq2.8} |h(s,\omega,x)-h(s,\omega,y)|
\leq C'|x-y|,\quad x,y\in\bar D, \omega\in\Omega.
\end{equation}
(see \cite[Lemma 3.1]{KPP}). }
\end{remark}

We  say that (\ref{eq2.6}) has the pathwise uniqueness property if
for any probability space
$(\bar{\Omega},\bar{\mathcal{F}},\bar{P})$ with filtration
$({\bar{\cal F}}_t)$  and any $\bar{X}_0\in\bar D$ and
$({\bar{\cal F}}_t)$-adapted semimartingale $\bar Z$  such that
${\cal L}({\bar{X}_0,\bar Z})={\cal L}({X_0,Z})$, ${P}[(\bar
X_t,\bar K_t)=(\bar X^{\prime}_t,\bar
K^{\prime}_t)\,,t\in\Rp\,]=1$ for any two  $(\bar{\cal F}_t)$
adapted strong solutions $(\bar X,\bar K),\,(\bar X^{\prime},\bar
K^{\prime})$ of (\ref{eq2.6}) on $(\,\bar{\Omega},\,\bar{\cal
F},\,\bar{\cal P})$.

We will consider the following assumptions.
\begin{enumerate}
\item[(F)] $f\in\ccc$, $\,\,f,f'f$ are bounded.
\item[($\Delta$)] $|\Delta Z|<\ro/\sup_{x\in\bar D}\|f(x)\|$.
\end{enumerate}

\begin{lemma} \label{lem2}
Assume (A),(B) and (F),($\Delta$). If $f,f^{\prime}f$ are locally
Lipschitz continuous,
  then the SDE (\ref{eq2.6})
has the pathwise uniqueness property.
\end{lemma}
\begin{proof}
Without loss of generality we  may consider solutions of
(\ref{eq2.6}) on  a probability space $(\Omega,\,\cal F,\cal P)$
with  a filtration $({\cal F}_t)$. Since $Z$ has bounded jumps, it
is a special semimartingale. Therefore $Z$ admits decomposition of
the form $Z_t=M_t+V_t$, $ t\in\Rp$, where  $M$ is an $({\cal
F}_t)$ adapted local martingale and $V$ is an $({\cal F}_t)$
predictable processes of bounded variation such that $M_0=V_0=0$
and $|\Delta V_t|\leq \ro/L$, $|\Delta M_t|\leq (2\ro)/L$, where
$L=\sup_{x\in\bar D}\|f(x)\|$. Let $( X, K),\,( X^{\prime},
K^{\prime})$ be two solutions of (\ref{eq2.6}). Set
\[
\tau_a=\inf\{t;\max(|X_t|,|X'_t|,|K|_t,|K'|_t,[M]_t,\langle
M\rangle_t, |V|_t)>a\},\quad a\in\N.
\]
We will show that $(X, K)=( X^{\prime}, K^{\prime})$ on any
interval $[0,\tau_a]$, $a\in\N$. Using the fact that
$|\Delta\int_0^tf(X_{s})\circ dZ_s|<r_0$  there is a constant
$C_a>0$ such that $\sup_{t\leq\tau_a}|X_t|$,
$\sup_{t\leq\tau_a}|X'_t|$, $|K|_{\tau_a}$, $|K'|_{\tau_a}$,
$[M]_{\tau_a}$, $\langle M\rangle_{\tau_a}$, $ |V|_{\tau_a}$ $\leq
C_a$. Moreover, $[Z]_{\tau_a}\leq2[M]_{\tau_a}+2|V|^2_{\tau_a}\leq
2C_a+2C_a^2$. By (\ref{eq2.5}) and Lipschitz continuity of
coefficients on $B(0,a)$, for any stopping time $\sigma$ we have
\begin{align*}
&E\sup_{t < \sigma}|X^{\tau_a}_t- X'^{\tau_a}_t|^{2} \leq C(2,a)
E\big([\int_0^\cdot(f(X_{s-})-f(X'_{s-}))\,dM^{\tau_a}_s]_{\sigma-}\\
&\qquad+\langle\int_0^\cdot
(f(X_{s-})-f(X'_{s-}))\,dM^{\tau_a}_s\rangle_{\sigma-}+ |
\int_0^\cdot
(f(X_{s-})-f(X'_{s-}))\,dV^{\tau_a}_s|_{\sigma-}\\
&\qquad+|\frac12\int_0^\cdot
(f^{\prime}(f(X_{s-})-f(X'_{s-}))\,d[Z^{\tau_a}]^c_s
|_{\sigma-}\\
&\qquad+|\int_0^\cdot
(h(s,\cdot,X_{s-})-h(s,\cdot,X'_{s-}))d[Z^{\tau_a}]^d_s|_{\sigma-}\big)\\
&\quad\leq
C(2,a)E\int_0^{\sigma-}\sup_{u<s}|X_u-X'_u|\,d([M^{\tau_a}]+\langle
M^{\tau_a}\rangle+|V^{\tau_a}|^2)_s.
\end{align*}
Using this and the stochastic version of Gronwall's lemma (see,
e.g. \cite[Lemma 2]{Ma} or \cite[Lemma C1]{Sl3}) we conclude that
$E\sup_{t }|X^{\tau_a}_t- X'^{\tau_a}_t|^{2}=0$, which implies
that $(X, K)=( X^{\prime}, K^{\prime})$ on $[0,\tau_a]$. Since
$\tau_a\nearrow\infty$, $P$-a.s., the lemma follows.
\end{proof}

\nsubsection{Stability of reflected Stratonovich SDEs}

Let $\{Z^n\}$  be a sequence of semimartingales defined possibly
on different probability spaces   $(\Omega^n, {\cal F}^n, {\cal
P}^n\,)$ and adapted to different filtrations $({\cal F}^n_t).$ We
will assume that $\{Z^n\}$ satisfies  condition  (UT) introduced
in Stricker \cite{St}, which  have appeared to be useful in the
theory of limit theorems for stochastic integrals and for
solutions of SDEs (see, e.g.,
 \cite{JMP,KP,MS,Sl,Sl1}).  We recall that $\{Z^n\}$ satisfies (UT) if
\begin{description}
\item[{(UT)}] for every  $q\in\Rp$ the
family of random variables
\[
\{\int_{[0,q]} V^n_s\,dZ^n_s;\,n\in\N\,,\,V^n\in\mbox{\bf
V}^n_q\}\quad \mbox{\rm is tight in $\R$,}
\]
where $ \mbox{\bf V}^n_q$  is the class of all discrete
predictable processes of the form $V^n_s=V^n_0+\sum_{i=0}^k
V^n_i\mbox{\bf 1}_{\{t_i<s\leq t_{i+1}\}}$  such that
$0=t_0<t_1<\ldots<t_k=q$ and every $V^n_i$  is  ${\cal F}^n_{t_i}$
measurable, $|V^n_i|\leq 1$  for every $i\in\N\cup\{0\}\,,\,n\in
\N,\,k\in\N.$
\end{description}

Let $\{(X^n,K^n)\}$ be  a sequence of  strong solutions of
(\ref{eq2.6}) driven by  $\{Z^n\}$,  i.e.
\begin{align}
X^n_t&=X_0^n+\int_0^tf(X^n_{s})\circ dZ^n_s+K^n_t\label{eq3.1}\\
&=X_0^n+\int_0^tf(X^n_{s-})\,dZ^n_s+\frac12
\int_0^tf^{\prime}f(X^n_{s-})\,d[Z^n]^c_s\nonumber\\
&\quad+\sum_{s\leq t}\{\vf(f\Delta
Z^n_s,X^n_{s-})-X^n_{s-}-f(X^n_{s-}) \Delta
Z^n_s\}+K^n_t,\nonumber\quad t\in\Rp.
\end{align}

We say that the  SDE (\ref{eq2.6}) has a weak solution if there
exists a probability space $(\,\wh{\Omega}\,,\wh{\cal
F}\,,\wh{\cal P}\,)$ with filtration $(\wh{\cal F}_t)$  satisfying
the usual conditions  and an $(\wh{\cal F}_t)$  adapted processes
$\wx,\wk,\wz$ such that ${\cal L}(\wh X_0,\wz)={\cal L}(X_0,Z)$
and  (\ref{eq2.6}) holds  for processes   $\wx,\wk,\wz$  instead
of $X,K,Z\,$. If ${\cal L}((\wx,\wk))= {\cal
L}((\wx^{\prime},\wk^{\prime}))$ for any two weak solutions
$(\wx,\wk)$, $(\wx^{\prime},\wk^{\prime})$ of (\ref{eq2.6}),
possibly defined on two different probability spaces, then we  say
that (\ref{eq2.6}) has the weak uniqueness property.

\begin{theorem}
\label{thm1} Assume (A),(B) and (F), ($\Delta$). Let $\{Z^n\}$ be
a sequence of ${\cal F}^n$  adapted semimartingales satisfying
(UT),  $\{X^n_0\}$ be a sequence of ${\cal F}^n_0$ measurable
random variables  such that $X^n_0\in\bar D$ for $n\in\N$ and let
\[
(X^n_0,Z^n)\arrowd(X_0,Z)\quad  in \,\,\Rd\times\D.
\]
If $\{(X^n,K^n)\}$  is a sequence of strong solutions of the  SDE
(\ref{eq3.1}), then
\begin{description}\item[{\rm (i)}] $\{X^n\}$ satisfies (UT) and
$\{|K^n|_q\}$  is bounded in probability, $q\in\Rp$,
\item[{\rm (ii)}] $\{(X^n,K^n))\}$  is tight in
$\Ddd$  and its every limit point  is a weak solution of
(\ref{eq2.6}),
\item[{\rm (iii)}]if moreover (\ref{eq2.6})  has the weak uniqueness
property and $(X,K)$ is its weak solution  then
\[
\,(X^n,K^n)\,\arrowd\,\,(X,K)\,\quad  in\quad \Ddd.
\]
\end{description}
\end{theorem}
\begin{proof} We follow the proof of \cite[Theorem 4]{Sl1}
and \cite[Theorem 5.2]{Sl3}. \\
(i) Set
\[
h^n(s,x)=\frac{\vf(\Delta Z^n_sf,x)-x-f(x)\Delta Z^n_s}{|\Delta
Z^n_s|^2},\quad (s,x)\in\Rp\times\bar D,\,n\in\N
\]
and observe that (\ref{eq3.1})   can be rewritten into
the form
\begin{align*}
X^n_t&=X_0^n+\int_0^tf(X^n_{s-})\,dZ^n_s+\frac12\int_0^tf^{\prime}f
(X^n_{s-})\,d[Z^n]^c_s\\
&\quad+\int_0^th^n(s,X^n_{s-})d[Z^n]^d_s+K^n_t.
\end{align*}
By Remark \ref{rem3}(b), $|h^n|\leq C$, $n\in\N$.  Since $\{Z^n\}$
satisfies (UT),  the sequence
\[
\{Y^n= X_0^n+\int_0^\cdot f(X^n_{s-})\,dZ^n_s+\frac12\int_0^\cdot f^{\prime}f
(X^n_{s-})\,d[Z^n]^c_s+\int_0^\cdot h^n(s,X^n_{s-})d[Z^n]^d_s\}
\]
satisfies (UT)  as well (see, e.g., \cite[Lemma 1.6]{MS}).
Assertion (i) now follows from \cite[Proposition 3]{Sl1}.\\
(ii) By  \cite{Jac}, $ (Z^n,[[Z^n]])\arrowd(Z,[[Z]])$ in
$\Ds{d(d+1)})$, which implies that
\begin{equation}\label{eq3.3}
\{(Z^n,[[Z^n]]^c,[Z^n]^d,[[Z^n]])\}\quad\mbox{\normalshape is tight in}\,\,
\Ds{d(1+2d)+1}).
\end{equation}
Using  (\ref{eq3.3}), part (i) and standard tightness criterions
for sequences satisfying (UT) (see, e.g., \cite[Proposition
3.3]{MS} one can show that $\{(Y^n,Z^n,[[Z^n]])\}$ is tight in
$\Ds{d(2+d)})$. From this and \cite[Proposition 4]{Sl1} we
conclude that
\[
\{(X^n,K^n,Z^n,[[Z^n]])\}\quad\mbox{\normalshape is tight
in}\,\,\Ds{d(3+d)}).
\]
Without loss of generality we may assume that
\[
(X^n,Z^n,K^n,[[Z^n]])\arrowd(\wx,\wk,\wz,[[\wz]])\quad\mbox{\normalshape in}
\,\, \Ds{d(3+d)}).
\]
The proof of (ii) will be completed once we show that $(\wx,\wk)$
is a solution of the reflecting  SDE of the  form
\begin{align*}
\wx_t&=\wx_0+\int_0^tf(\wx_{s-})\,d\wz_s
+\frac12\int_0^tf^{\prime}f(\wx_{s-})\,d[\wz]^c_s\\
&\quad+\int_0^t h(s,\wx_{s-})\,d[\wz]^d_s+\wk_t,\quad
t\in\Rp.\nonumber
\end{align*}
while by \cite[Proposition 4]{Sl1}, to show the last statement it
sufficies to prove the convergence  $(X^n,K^n,Y^n)\arrowd(\wx,\wk,\wy)$ in
$\Ds{3d})$, where
\begin{equation}\label{eq3.4}\wy_t=\wx_0+\int_0^tf(\wx_{s-})\,d\wz_s
+\frac12\int_0^tf^{\prime}f(\wx_{s-})\,d[\wz]^c_s
+\int_0^t h(s,\wx_{s-})\,d[\wz]^d_s,\quad t\in\Rp.
\end{equation}
Let  $\{\epsilon_k\}$  be a sequence of  constants
 such that $\epsilon_k\downarrow0$ and
$P[|\Delta Z_t|=\epsilon_k;t\in\Rp]=0,\,k\in\N$. Set
$\,J^{n,\epsilon_k}_t=\sum_{0<s\leq t}\Delta Z^n_s \mbox{\bf
1}_{\{|\Delta Z^n_s|>\epsilon_k\,\}}$,
$\,\wj^{\epsilon_k}_t=\sum_{0<s\leq t}\Delta Z_s\mbox{\bf
1}_{\{|\Delta \wz_s|>\epsilon_k\,\}}$,
$Z^{n,\epsilon_k}=Z^n_t-J^{n,\epsilon_k}_t$ and
$\wz^{\epsilon_k}=\wz_t-\wj^{\epsilon_k}_t$, $k,n\in\N.$ Since
\begin{equation}
\label{eq3.5}
(X^n,Z^n,[[Z^{n,\epsilon_k}]],[J^{n,\epsilon_k}])\arrowd
(\wx,\wz,[[\wz^{\epsilon_k}]],[\wj^{\epsilon_k}])
\quad\mbox{\normalshape in}\,\,\Ds{d(2+d)+1})
\end{equation}
and $\sum_{s\leq t}|\Delta \wz_s|^2\mbox{\bf 1}_{\{|\Delta
Z^n_s|>\epsilon_k\}}\nearrow [\wz]^d_t$, $P$-a.s. for $t\in\Rp$,
we can find a sufficiently slowly increasing sequence
$k_n\uparrow+\infty$  such that
\begin{equation}
\label{eq3.6}
\quad(X^n,Z^n,[[Z^{n,\epsilon_{k_n}}]],[J^{n,\epsilon_{k_n}}])\arrowd
(\wx,\wz,[[\wz]]^c,[\wz]^d)\,\,\mbox{\normalshape in}\,\,\Ds{d(2+d)+1}).
\end{equation}
Set $\phi(y,x)=\varphi(fy,x)-x-f(x)y$, $x,y\in\bar D$ and observe
that by  (\ref{eq3.6}),
\begin{align}
&
(X^n,K^n,Z^n,[[Z^{n,\epsilon_{k_n}}]],[J^{n,\epsilon_{k_n}}],\sum_{s\leq
\cdot} \phi(\Delta Z^n_s,X^n_{s-})
\mbox{\bf 1}_{\{|\Delta Z^n_s|>\epsilon_k\}})\label{eq3.7}\\
&\quad\quad\arrowd (\wx,\wk,\wz,[[\wz]]^c,[\wz]^d,\sum_{s\leq
t}\phi(\Delta \wz_s,\wx_{s-}) \mbox{\bf 1}_{\{|\Delta
\wz_s|>\epsilon_k\}})\,\, \mbox{\normalshape
in}\,\,\Ds{d(4+d)+1}). \nonumber
\end{align}
By Remark \ref{rem3}(b),  $\sum_{s\leq t} |\phi(\Delta
Z^n_s,X^n_{s-})| \mbox{\bf 1}_{\{\epsilon_k\geq|\Delta
Z^n_s|>\epsilon_{k_n}\}} \leq C\sum_{s\leq t} |\Delta
Z^n_s|^2\mbox{\bf 1}_{\{\epsilon_k\geq|\Delta Z^n_s|\}}$. Since
\[
\sum_{s\leq\cdot}|\Delta Z^n_s|^2 \mbox{\bf
1}_{\{\epsilon_k\geq|\Delta Z^n_s|\}} \arrowd \sum_{s\leq
\cdot}|\Delta \wz_s|^2 \mbox{\bf 1}_{\{\epsilon_k\geq|\Delta
\wz_s|\}} \quad\mbox{\normalshape in}\,\,\Dj
\]
and $\sum_{s\leq t}|\Delta \wz_s|^2 \mbox{\bf
1}_{\{\epsilon_k\geq|\Delta \wz_s|\}}\searrow0$, $P$-a.s. for
$t\in\Rp$, it follows from (\ref{eq3.7})  that
\begin{align}
\label{eq3.8} & (X^n,K^n,Z^n,[[Z^{n,\epsilon_{k_n}}]],\sum_{s\leq
\cdot} \phi(\Delta Z^n_s,X^n_{s-})
\mbox{\bf 1}_{\{|\Delta Z^n_s|>\epsilon_{k_n}\}})\\
&\quad\quad\arrowd
(\wx,\wk,\wz,[[\wz]]^c,\sum_{s\leq\cdot}\phi(\Delta
\wz_s,\wx_{s-}))\quad \mbox{\normalshape in}\,\,\Ds{d(4+d)}).
\nonumber
\end{align}
By the above and \cite[Theorem 2.6]{JMP},
\begin{eqnarray*}
\lefteqn{(\,X^n,K^n,\int_0^{\cdot}
f(X^n_{s-})dZ^n_s,\int_0^{\cdot}f^{\prime}f(X^n_{s-})
d[Z^{n,\epsilon_{k_n}}]_s,\sum_{s\leq\cdot}\phi(\Delta Z^n_s,X^n_{s-})
\mbox{\bf 1}_{\{|\Delta Z^n_s|>\epsilon_{k_n}\}}\,)}\\
& &\arrowd
(\wx,\wk,\int_0^{\cdot}f(\wx_{s-})d\wz_s,\int_0^{\cdot}f^{\prime}f(\wx_{s-})
d[\wz]^c,
\sum_{s\leq \cdot}\phi(\Delta \wz_s,\wx_{s-}))\quad
\mbox{\rm in}\,\,\Ds{4d}).
\end{eqnarray*}
Finally, arguing as in the proof of \cite[Theorem 5.2]{Sl3}
we show that \[\frac12\int_0^{\cdot}f^{\prime}f(X^n_{s-})d[Z^n]^c_s +
  \sum_{s\leq\cdot}\phi(\Delta Z^n_s,X^n_{s-})
\mbox{\bf 1}_{\{|\Delta Z^n_s|\leq\epsilon_{k_n}\}}-\frac12\int_0^{\cdot}f^{\prime}f(X^n_{s-})
d[Z^{n,\epsilon_n}]_s\arrowp0.
\]
Consequently, $(X^n,Y^n,K^n)\arrowd(\wx,\wy,\wk)$ in $\Ds{3d})$,
where $\wy$ satisfies (\ref{eq3.4}).\\
(iii) Follows immediately from (ii).
\end{proof}

\begin{corollary}
\label{cor1} Assume (A), (B) and  (F), ($\Delta$). Let  $Z$  be
an  $({\cal F}_t)$  adapted semimartingale and $X_0$ be ${\cal
F}_0$ measurable random variable such that $X_0\in\bar D$. Then there exists a weak solution
of   (\ref{eq2.6}).
\end{corollary}
\begin{proof}
Let  $\{T_n\}$  be a sequence of partitions of  $\Rp$ satisfying
(\ref{eq2.1}).  Set ${\cal F}\rn_t={\cal F}_{\rho^n_t}$. Then
$Z\rn $ is an ${\cal F}\rn$ adapted semimartingale. By the
Bichteler-Mokobodski theorem the sequence $\{Z\rn\}$ satisfies
(UT).  Assume that $\{(X^n,K^n)\}$ is a family of strong solutions
of (\ref{eq3.1}) driven by $\{Z\rn\}$, that is
\begin{align*}
X^n_t&=X_0+\int_0^tf(X^n_{s-})\,dZ\rn_s+\frac12
\int_0^tf^{\prime}f(X^n_{s-})\,d[Z\rn]^c_s\\
&\quad+\sum_{s\leq t}\{\vf(f\Delta
Z\rn_s,X^n_{s-})-X^n_{s-}-f(X^n_{s-})
\Delta Z\rn_s\}+K^n_t\\
&=X_0+\sum_{s\leq t}\{\vf(f\Delta
Z\rn_s,X^n_{s-})-X^n_{s-}\}+K^n_t, \quad\,t\in\Rp,\,n\in\N.
\end{align*}
Simple calculation show that $(X^n,K^n)$  are defined by the
following recurent formula: $X^n_0=X_0$, $K^n_0=0$,
\begin{equation}
\left\{\begin{array}{ll}
\Delta Y^n_{\tkn}&=\vf(f(Z_{\tkn}-Z_{\tk}))-X^n_{\tk}\,, \\
X^n_{\tkn}&=\Pi_{\bar D}\big(X^n_{\tk}+\Delta Y^n_{\tkn}\big)
=\Pi_{\bar D}\big(\vf(f(Z_{\tkn}-Z_{\tk})),\big), \label{eq3.9}\\
K^n_{\tkn}&=K^n_{\tk}+(X^n_{\tkn}-X^n_{\tk})-\Delta Y^n_{\tkn}
\end{array}
\right.
\end{equation}
and $X^n_t=X^n_{\tk}$, $K^n_t=K^n_{\tk}$,  $t\in[\tk,\tkn)$,
$k\in\N\cup\{0\}$.  Since  $Z\rn\lra Z$ a.s.  in  $\D$, the result
is an immediate  consequence of Theorem \ref{thm1}(i).
\end{proof}
\begin{theorem}
\label{thm2} Assume (A), (B) and (F), ($\Delta$). Let $Z$ be an
$({\cal F}_t)$  adapted semimartingale and $X_0$ be a ${\cal F}_0$
measurable and such that $X_0\in\bar D$. If $f,f^{\prime}f$ are
 locally Lipschitz continuous,  then there exists a unique strong
solution of the  SDE (\ref{eq2.6}).
\end{theorem}
\begin{proof}
Let  $\{\{\tau^n_k\}\}$  be an array of   $({\cal F}_t)$ stopping
times such that $\tau^n_0=0$, $\tau^n_k=\inf\{t>\tau^n_{k-1};
|\Delta Z_t|>\frac1{n}\}\wedge (\tau^n_{k-1}+\frac1{n})$,
$n,k\in\N$.  Let us consider the sequence  $\{Z^n\}$ of
 $({\cal F}_t)$ adapted semimartingales of the form: $Z^n_0=0$ and
\[
Z^n_t=Z^n_{\tau_k^n},\quad t\in[\tau_k^n,\tau_{k+1}^n),\quad
n,k\in\N.
\]
It is easy to see that
\begin{equation}
\label{eq3.10} \sup_{t\leq q}|Z^n_t-Z_t|\lra0,\quad
P\mbox{-a.s.},\,q\in\Rp.
\end{equation}
Observe that $Z^n$ is an $({\cal F}_t)$ adapted semimartingale and
the solution $(X^n,K^n)$ of (\ref{eq3.1}) driven by $Z^n$ has the
following form: $X^n_0=X_0$, $K^n_0=0$,
\begin{equation}
\left\{\begin{array}{ll}
\Delta Y^n_{\tau_{k+1}^n}&= \vf(f(Z_{\tau_{k+1}^n}-Z_{\tau^n_k}))-X^n_{\tau^n_k}, \\
X^n_{\tau_{k+1}^n}&=\Pi_{\bar D}
\big(\vf(f(Z_{\tau_{k+1}^n}-Z_{\tau^n_k}))\big),\label{eq3.11}\\
K^n_{\tau_{k+1}^n}&=K^n_{\tau_k^n}+(X^n_{\tau_{k+1}^n}-X^n_{\tau^n_k})-\Delta
Y^n_{\tau_{k+1}^n}
\end{array}
\right.
\end{equation}
and $X^n_t=X^n_{\tau_k^n}$, $K^n_t=K^n_{\tau_k^n}$,
$t\in[\tau_k^n,\tau_{k+1}^n)$, $k\in\N\cup\{0\}$.
 We will show that  $\{(X^n,K^n)\}$
converges in probability. For this purpose it suffices to show
that from any subsequences $(l)\subset(n), (m)\subset(n) $ it is
possible to choose further subsequences
$(l_k)\subset(l),(m_k)\subset(m)$ such that
\[(
(X^{l_k},K^{l_k}),( X^{m_k},K^{m_k}))
\arrowd((\wx,\wk),(\wx',\wk'))\quad\mbox{\rm in}\,\, \Ds{4d})
\]
(see Gy\"ongy and Krylov \cite{GK}). Using (\ref{eq3.10}), the fact
that $\{Z^n\}$ satisfies (UT) and arguments from  the  proof of
Theorem \ref{thm1}(i) we show  that
\[    \{(
X^{l},K^l,Z^l, X^{m},K^m,Z^m)\}\quad \mbox{\rm is tight in\,\,} \Ds{6d}).
\]
Therefore we can choose  subsequences
$(l_k)\subset(l),(m_k)\subset(m)$ such that
\[(
X^{l_k},K^{l_k},
Z^{l_k},X^{m_k},K^{m_k},Z^{m_k})\arrowd(\wx,\wk,\wz,\wx',\wk',\wz),\quad
\mbox{\rm  in}\,\,\Ds{6d}),
\]
where  $\wz$  is a  semimartingale with respect to the natural
filtration ${\cal F}^{\wx,\wk,\wx',\wk',\wz}$ such that ${\cal
L}(\wx_0,\wx)={\cal L}(X_0,Z)$. By the  arguments from the proof
of Theorem \ref{thm1}(i) the processes $(\wx,\wk)$ and
$(\wx'\wk')$ are two solutions of (\ref{eq2.6}) with $\wx_0,\wz$
instead of $X_0,Z$. Since by Lemma \ref{lem2} the SDE
(\ref{eq2.6}) is pathwise unique, $(\wx,\wk)=(\wx'\wk')$.
Consequently, $\{(X^n,K^n)\}$  converges in probability in  $\Ddd$
to some pair of  processes  $(X,K)$. Clearly,    $(X,K)$ is a
strong solution of (\ref{eq2.6}), so using once again the pathwise
uniqueness property we conclude that $(X,K)$  is a unique strong
solution of (\ref{eq2.6}).
\end{proof}

\begin{theorem}
\label{thm3}Assume (A),(B) and (F), ($\Delta$). Let  $\{Z^n\}$  be
a sequence of  $({\cal F}^n_t)$  adapted semimartingales
satisfying  (UT)   and $\{X^n_0\}$ be a sequence of ${\cal F}^n_0$
measurable random variables  such that $X^n_0\in\bar D$, $n\in\N$.
If   $\{(X^n,K^n)\}$  is a  sequence of strong solutions of the
SDE (\ref{eq3.1}) and  $f,f^{\prime}f$ are  locally
Lipschitz continuous, then the following two implications
are true:
\begin{description}
\item[{\rm (i)}] if  $X^n_0\arrowp X_0$  and  $Z^n\arrowp Z$  in
$\D$
   then
\[(X^n,K^n,Z^n)\arrowp(X,K,Z)
\quad\mbox{\rm{ in}} \,\, \Ds{2d}),
\]
\item[{\rm (ii)}] if  $X^n_0\arrowp X_0$  and   $\sup_{t\leq
q}|Z^n_t-Z_t|\arrowp0,$  $q\in\Rp$,  then
\[
\sup_{t\leq q}(|X^n_t-X_t|+|K^n_t-K_t|)\arrowp0,\quad q\in\Rp,
\]
\end{description}
where $(X,K)$ is the unique strong solution of (\ref{eq2.6}).
\end{theorem}
\begin{proof}
(i) We use Theorem \ref{thm2}  and follow the proof of
\cite[Corollary 11(i)]{Sl1}.\\
(ii) Let us note that
\[
\Delta X_t=\vf(\Delta Z_tf,X_{t-})-X_{t-}+\Delta K_t
\]
and if  $\Delta X_t\neq0$  then    $\Delta Z_t\neq0$. By  part (i)
and \cite[Corollary C]{Sl},
\[
\sup_{t\leq q}|X^n_t-X_t|\arrowp0,\quad q\in\Rp,
\]
and the proof is complete.
\end{proof}

\begin{corollary}
\label{cor2}Assume  (A),(B) and  (F), ($\Delta$). Let  $Z$  be an
$({\cal F}_t)$  adapted semimartingale and $X_0$ be a ${\cal F}_0$
measurable and such that $X_0\in\bar D$.  Let  $\{T_n\}$ be a
sequence of partitions of $\Rp$  satisfying (\ref{eq2.1})  and
$\{(X^n,K^n)\}$  denotes a sequence of solutions of (\ref{eq3.1})
corresponding to  $Z^n=Z\rn$,  $X^n_0=X_0$, $n\in\N$.  If
$f,f^{\prime}f$ are locally Lipschitz continuous, then
\[
\sup_{t\leq
q}(|X^n_t-X_t\rn|+|K^n_t-K_t\rn|)\arrowp0,\quad\,q\in\Rp,
\]
where $(X,K)$ is a unique strong solution of (\ref{eq2.6}).
\end{corollary}
\begin{proof}
By Theorem \ref{thm3}(i),
\[
(X^n,K^n,Z\rn)\arrowp(X,K,Z)\quad\mbox{\rm in}\,\,\Ds{3d}).
\]
Consequently,
\[
(X^n,X\rn)\arrowp(X,X)\,\,\mbox{\rm in}\,\,\Ds{2d})\quad\mbox{\rm and}
\quad (K^n,K\rn)\arrowp(K,K)\,\,\mbox{\rm in}\,\,\Ds{2d}),\]
which implies  the corollary.
\end{proof}

\nsubsection{Approximations of  Wong-Zakai type}

Let us consider a sequence $\{T_n\}$  of partitions of  $\Rp$
satisfying  condition  (\ref{eq2.1}). We will approximate in
probability the first coordinate  $X$ of the solution of
(\ref{eq2.6}) by solutions of nonreflected SDEs given by the
following recurrent scheme: $\wh X^n_0=X_0$,
\[
\wh X^n_t=\Pi_{\bar D}(\wh
X^n_{\tk-})+(\tkn-\tk)^{-1}\int_{\tk}^tf(\wh
X_s^n)\,ds(Z_{\tkn}-Z_{\tk}), \quad
t\in[\tk,\tkn),\,k\in\N\cup\{0\}.
\]
One can observe that if  $\bar Z^n$ is the  linear approximation
of $Z$ of the form $\bar Z^n_0=Z_0=0$,
\[
\bar Z^n _t=Z_{\tk}+\frac{t-\tk}{\tkn-\tk}(Z_{\tkn}-Z_{\tk}),
\quad t\in[\tk,\tkn),\,n\in\N,k\in\N\cup\{0\},
\]
then $\wh X^n$ satisfies the equation
\[
\wh X^n_t=X_0+\int_0^tf(\wh X^n_s)\,d\bar Z^n_s+\wh K^n_t,\quad n\in\N,
\]
where $\wh K^n_t=\sum_{k;t^n_k\leq t} (\Pi_{\bar D}(\wh
X^n_{\tk-})-\wh X^n_{\tk-})$, $n\in\N$. Moreover, if $(X^n,K^n)$
is the strong solution of (\ref{eq3.1})  corresponding to  $Z\rn$
then $\wh X^n_t=X^n_t$ for  $t\in T_n$  and $\wh K^n_t=K^n_t$,
$t\in\Rp$.
\begin{theorem}
\label{thm4}Assume  (A), (B) and  (F), ($\Delta$). If
$f,f^{\prime}f$ are  locally Lipschitz continuous, then
\begin{description}
\item[{\rm (i)}]
$\displaystyle{
\sup_{t\leq q,\,t\in T_n}|\wh X^n_t-X_t|\arrowp0,\quad q\in\Rp}$  and
\[
\wh X^n_t\arrowp X_t
\]
provided  that $\Delta Z_t=0$  or $t\in\liminf_{n\rightarrow+\infty}T_n$,
\item[{\rm (ii)}]$\displaystyle{\wh X^n\arrowp X}$ in the $S$-topology,
\item[{\rm (iii)}]if  moreover $Z$  is a semimartingale  with continuous
trajectories then
\[
\sup_{t\leq q}|\wh X^n_t-X_t|\arrowp0,\quad q\in\Rp,
\]
\end{description}
where  $X$  is the first coordinate of the unique strong solution
of the  SDE  (\ref{eq2.6}).
\end{theorem}
\begin{proof}
(i) Assume that  $\{(X^n,K^n)\}$ is a family of  strong solutions
of (\ref{eq3.1})  corresponding to  $\{Z\rn\}$. Since $\wh
X\rnn=X^n$, it follows from Corollary (\ref{cor2}) that
\[
\sup_{t\leq q,\,t\in T_n}|\wh X^n_t-X_t| =\sup_{t\leq q}|\wh
X\rnn_t-X\rn_t|\arrowp0,\quad q\in\Rp.
\]
Moreover, if $\Delta Z_t=0$ then $X\rn_t\to X_t$.  Therefore  for
$q\geq t$  and $\tk=\rho_n(t)\leq t$ we have
\[
|\wh X^n_t-X_t|\leq \sup_{t\leq q,\,t\in T_n}|\wh X^n_t-X_t|
+L|Z_{\tkn}-Z_{\tk}|+|X\rn_t-X_t|\lra0,\quad P\mbox{-a.s.},\]
where   $L=\sup_{x\in\bar D}\|f(x)\|$.\\
(ii) We will show that $\{\wh X^n\}$  is tight in the $S$
topology. By Theorem \ref{thm1}(i) the sequence  $\{|\wh
K^n|_q=|K^n|_q\}$ is bounded in probability. Therefore it suffices
to show that
\[
\{\wh Y^n=\int_0^{\cdot}f(\wh X^n_s)\,d\bar Z^n_s\} \quad
\mbox{\rm satisfies (UT)}
\]
and then use \cite[Theorem 4.1]{Jak}. Set
$\rho^{n,*}_t=\min\{t^n_k;\,t^n_k>t\}$ and $\wh{\cal F}^n_t={\cal
F}_{\rho^{n,*}_t}={\cal F}_{\tk}$, $t\in[\tkk,\tk)$,
$k\in\N\cup\{0\}$, $n\in\N$. Clearly, $\wh Y^n$ is an $(\wh{\cal
F}^n_t)$ adapted processes admitting the decomposition into the
sum of two $(\wh{\cal F}^n_t)$ adapted processes of the form
\[
\int_0^{\cdot}f(\wh X\rnn_{s-})\,d\bar Z^n_s
+\int_0^{\cdot}(f(\wh X^n_s)-f(\wh X\rnn_{s-}))\,d\bar Z^n_s=I^{n,1}+I^{n,2}.
\]
Since
\[I^{n,1}_t=\sum_{k;\tk\leq t}f(\wh X^n_{\tkk})
(Z_{\tk}-Z_{\tkk})+(t-\rho^n_t)f(\wh X^n_{\rho^n_t})
(Z_{\rho^{n,*}_t}-Z_{\rho^n_t}),\quad t\in\Rp,
\]
it  follows from the Bichteler-Mokobodski theorem that
$\{I^{n,1}\}$ satisfies (UT). On the other hand, in view of
boundedness of $f'f$ there is $C>0$ such that for all sufficiently
large $n$,
\begin{align*}
|I^{n,2}|_q\leq&\int_{0}^{\rho^{n,*}_q}\|f(\wh X^n_s)-f(\wh
X^{n,\rho^n}_{s-}))|\|\,d|\bar Z^n|_s\\&=\sum_{k;\,\tk\leq
q}\int_{\tkk}^{\tk}\|f(\wh X^n_s)-f(\wh X^{n}_{\tkk})\|\,d|\bar
Z^n|_s\\
&\leq\sum_{k;\,\tk\leq \rho^{n,*}_q}
\int_{\tkk}^{\tk}\int^s_{\tkk}\|f'f(\wh X^n_u)\|\,d|\bar
Z^n|_u\,d|\bar Z^n|_s\\&\leq C \sum_{k;\,\tk\leq
\rho^{n,*}_q}|Z_{\tk}-Z_{\tkk}|^2\leq C \sum_{k;\,\tk\leq
q+1}|Z_{\tk}-Z_{\tkk}|^2.
\end{align*}
This implies that $\{|I^{n,2}|_q\}$  is bounded in probability and
completes the proof of (ii). \\
(iii) In this case the solution  $X$  has  continuous trajectories
as well. Therefore $ \sup_{t\leq q}|X_t-X^{\rho_n}_t|\lra0$,
$P$-a.s, $q\in\Rp$. Similarly, since the trajectories of $Z $ are
continuous, we have  $\sup_{t\leq q}|\Delta Z\rn_t|\lra0$,
$P$-a.s., $q\in\Rp$, which  implies that for all sufficiently
large $n$,
\begin{equation}
\label{eq4.1}
\sup_{t\leq q}|\wh X^n_t-\wh X^{n,\rho_n}_t |\leq L\sup_{t\leq
q+1}|\Delta Z\rn_t|\lra0,\quad P\mbox{-a.s.}
\end{equation}
Now,  (iii)  is an easy consequence of (i).
\end{proof}
\medskip

We now consider standard Wong-Zakai type  approximation of
reflected Stratonovich SDEs. The approximation processes $(\bar
X^n,\bar K^n)$ are defined by the recurrent scheme: $\bar
X^n_0=X_0$ and
\begin{equation}\label{eq4.20}
\bar X^n_t=\bar X^n_{\tk}+ (\tkn-\tk)^{-1}\int_{\tk}^tf(\bar
X^{n}_{s})\,ds(Z_{\tkn}-Z_{\tk})+\bar K^n_t-\bar K^n_{\tkn}
\end{equation}
for $t\in[\tk,\tkn)$, $k\in\N\cup\{0\}$. Clearly, $(\bar X^n,\bar
K^n)$ is a solution of the Skorokhod problem associated with
\[
\bar Y^n=X_0+\int_0^\cdot f(\bar X^n_s)\,d\bar Z^n_s,\quad
n\in\N.
\]

\begin{theorem}
\label{thm5}Assume (A),(B) and (F). If $f,f^{\prime}f$ are locally
Lipschitz continuous and $Z$ is a semimartingale with continuous
trajectories then
\[
\sup_{t\leq q}(|\bar X^n_t-X_t|+|\bar K^n_t-K_t|)\arrowp0,\quad q\in\Rp.
\]
\end{theorem}
\begin{proof}
We will use the notation from the proof of Theorem \ref{thm4}. For
brevity,  we write $\Delta Z_k=Z_{\tk}-Z_{\tkk}$, $\Delta
t_k=(\tk-\tkk)$, $k\in\N$. From Theorem \ref{thm4},  we know that
$\sup_{t\leq q}|\wh X^n_t-X_t|\arrowp0$, $q\in\Rp$. Moreover,
$(\wh X\rnn,\wh K^n)=(\wh X\rnn,\wh K\rnn)=(X^n,K^n)$ is a solution of the Skorokhod problem
associated with $Y^n=X_0 +\int_0^{\rho^n_{\cdot}}f(\wh X^n_s)\,d
\bar Z^n_s$ and
\begin{equation}\label{eq4.2}
|\Delta K^n_t|\leq|\Delta Y^n_t|\leq \sup_x\|f(x)\|\,|\Delta
Z\rn_t|, \quad t\in\Rp.
\end{equation}
Since $\{Y^n=X^n-K^n\}$ converges uniformly in probability, it
follows from \cite[Proposition 3]{Sl1} that
\begin{equation}\label{eq4.3}
\{|K^n|_q\}\quad\mbox{\rm is bounded in probability}\,\,q\in\Rp.
\end{equation}
In the rest of the proof we will show that
\begin{equation}\label{eq4.4}
\sup_{t\leq q}|\wh X^n_t-\bar X^n_t|\arrowp0,\quad q\in\Rp.
\end{equation}
We start with proving that (\ref{eq4.3}) holds with $K^n$ replaced
by  $\bar K^n$. To this end, we  decompose $Y^n$ into the sum
\begin{align*}\bar Y^n_t&=(\bar Y^n_t-\bar
Y^{n,\rho^n}_t)+\int_0^{\rho^n_t}f(\bar X\rnn_{s-})\,d
Z\rn_s+\int_0^{\rho^n_t}(f(\bar X^n_s)-f(\bar
X^{n,\rho^n}_{s-}))\,d\bar Z^n_s\\
&=I^{n,1}_t+I^{n,2}_t+I^{n,3}_t.
\end{align*}
Since $\{Z\rn\}$ satisfies (UT) and $\sup_{t\leq q}|Z\rn_t-Z_t|\to
0$, $P$-a.s., $q\in\Rp$,  and since moreover  $f$ is bounded, we
see that for all sufficiently large $n$,
\begin{equation}
\label{eq4.5} \sup_{t\leq q}|I^{n,1}_t|=\sup_{t\leq q}|\bar
Y^n_t-\bar Y\rnn_t|\leq L\sup_{t\leq q+1} |\Delta Z\rn_t|\lra
0,\quad P\mbox{-a.s.},\,\,q\in\Rp
\end{equation}
and
\begin{equation}
\label{eq4.6}
\{I^{n,2}=\int_0^{\rho^n_{\cdot}}f(\bar
X\rnn_{s-})\,d Z\rn_s\}\quad\mbox{\rm is  C-tight}.
\end{equation}
Furthermore,
\begin{align*}
I^{n,3}_t=\int_{0}^{\rho^n_t}(f(\bar X^n_s)-f(\bar
X^{n,\rho^n}_{s-}))\,d\bar Z^n_s&=\sum_{k;\tk\leq
t}\int_{\tkk}^{\tk}(f(\bar X^n_s)-f(\bar X^{n}_{\tkk}))\,d\bar
Z^n_s\\
&=\sum_{k;\tk\leq t}\int_{\tkk}^{\tk}\int^s_{\tkk}f'(\bar
X^n_u)\,d\bar X^n_u\,d\bar Z^n_s.
\end{align*}
Since $f'f$ is bounded, it follows from Lemma \ref{lem1} that
there is $C>0$ such that
\begin{align*}
|\int^s_{\tkk}f'(\bar X^n_u)\,d\bar
X^n_u|&\leq\int^s_{\tkk}\|f'f(\bar X^n_u)\|\,du |\frac{|\Delta
Z_k|}{\Delta t_k}+|\int^s_{\tkk}f'(\bar
X^n_u)\,d\bar K^n_u|\\
&\leq 2\int^s_{\tkk}\|f'f(\bar X^n_u)\|\,du |\frac{|\Delta
Z_k|}{\Delta t_k}\leq C|\Delta Z_k|.
\end{align*}
It follows that for any $s<t$,
\[|I^{n,3}_t-I^{n,3}_s|\leq C \sum_{k;\,s<\tk\leq
t}|Z_{\tk}-Z_{\tkk}|^2=C([Z\rn]_t-[Z\rn]_s).\] From the above and
the fact that $\sup_{t\leq q}|[Z\rn]_t-[Z]_t|\arrowp0$, $q\in\Rp$,
we deduce that
\begin{equation}\label{eq4.7}
\{I^{n,3}=\int_{0}^{\rho^n_{\cdot}}(f(\bar X^n_s)-f(\bar
X^{n,\rho^n}_{s-}))\,d\bar Z^n_s\}\quad\mbox{\rm is C-tight}.
\end{equation}
Combining (\ref{eq4.7})  with (\ref{eq4.6})  and (\ref{eq4.5}) we
see that  $\{\bar Y^n=I^{n,1}+I^{n,2}+I^{n,3}\}$ is C-tight.
Therefore, by  \cite[Proposition 3,4]{Sl1},
\begin{equation}\label{eq4.8}
\{|\bar K^n|_q\}\quad\mbox{\rm is bounded in
probability}
\end{equation}
and
\begin{equation}\label{eq4.9}
\{(\bar X^n,\bar K^n)\}\quad\mbox{\rm
is  C-tight. }
\end{equation}
By (\ref{eq4.1})  and (\ref{eq4.9}),  to prove (\ref{eq4.4}) it
suffices now to show that
\begin{equation}\label{eq4.10}
\sup_{t\leq q}|\wh X\rnn_t-\bar X\rnn_t|\arrowp0,\quad q\in\Rp.
\end{equation}
To check this we will use the stochastic  Gronwall inequality.
Since $Z$ is a continuous semimartingale, it admits decomposition
$Z=M+V$, where $M$ is a continuous locally square integrable
martingale  such that $M_0=0$ and $V$ is a continuous predictable
process with bounded variation such that $V_0=0$. Set
\[
\sigma^n_a=\inf\{t;\max(|M_t,|V|_t|,|\bar X^n_t|,|\bar K^n|_t)>a\},
\quad a\in\Rp,\,n\in\N.
\]
Obviously, $\lim_{a\rightarrow\infty}\limsup_{n\rightarrow\infty}
{P}(\sigma^n_a\leq q)=0$, $q\in\Rp$. Therefore in each step $n$ we
can restrict  our attention to processes stopped at $\sigma^n_a$.
Now, for $n\in\N$, $b>0$ set
\[
\tau^n_b=\inf\{ t>0; \max([M\rn]_t,\langle M\rn\rangle_t,
|V\rn|^2_t,|\wh X\rnn_t|,|\wh K\rnn|_t)
>b\}.\]
The processes $[M\rn],\langle M\rn\rangle, |V\rn|^2$, $|\wh
X\rnn_t|,|\wh K\rnn|_t$ stop\-ped at $\tau^n_b$  are bounded.
Moreover,
\[
\lim_{b\rightarrow\infty}\limsup_{n\rightarrow\infty}{
P}(\tau^n_b\leq q)=0,\quad q\in\Rp,
\]
so we can restrict our attention to the processes stopped at
$\tau^n_b$. By \cite[Lemma 2.3(i)]{Sa},
\begin{align*}
|\bar X\rnn_t-\wh X\rnn_t|^2&\leq |\bar Y\rnn_t-\wh Y\rnn_t|^2
+\frac{1}{\ro}\int_0^{\rho^n_t}|\bar X^n_s-\wh X\rnn_s|^2\,d(|\bar K^n|
+|\wh K\rnn|)_s\\
&\quad+2\int_0^{\rho^n_t}
(\bar Y\rnn_t-\bar Y^n_s)-(\wh Y\rnn_t-\wh Y\rnn_s)\,d(\bar K^n-\wh K\rnn)_s\\
&=|\bar Y\rnn_t-\wh Y\rnn_t|^2
+\frac{1}{\ro}\int_0^{t}|\bar X\rnn_{s-}-\wh X\rnn_{s-}|\,
d(|\bar K\rnn|+|\wh K\rnn|)_s\\
&\quad +2\int_0^{t}(\bar Y\rnn_t-\bar Y\rnn_s)-(\wh Y\rnn_t-\wh
Y\rnn_s)\, d(\bar K\rnn-\wh K\rnn)_s +R^{n,1}_t,
\end{align*}
where
\begin{align*}
R^{n,1}_t=&\frac{1}{\ro}\int_0^{\rho^n_t}
|(\bar X^n_s-\bar X\rnn_{s-})-(\wh X\rnn_s-\wh X\rnn_{s-})|^2\,
d(|\bar K^n|+|\wh K\rnn|)_s\\
&\qquad+2\int_0^{\rho^n_t}(\bar Y\rnn_s-\bar Y^n_s)\,d(\bar K^n-\wh K\rnn)_s.
\end{align*}
By the integration by parts formula,
\begin{align*}
&2\int_0^{t}(\bar Y\rnn_t-\bar Y\rnn_s)-(\wh Y\rnn_t-\wh Y\rnn_s)\,
d(\bar K\rnn-\wh K\rnn)_s\\
&\quad =2\int_0^t(\bar K\rnn_{s-}-
\wh K\rnn_{s-})\,d(\bar Y\rnn-\wh Y\rnn)_s\\
&\quad=2\int_0^t(\bar X\rnn_{s-}- \wh X\rnn_{s-})\,d(\bar
Y\rnn-\wh Y\rnn)_s-2\int_0^t(\bar Y\rnn_{s-}-
\wh Y\rnn_{s-})\,d(\bar Y\rnn-\wh Y\rnn)_s\\
&\quad=2\int_0^t(\bar X\rnn_{s-}- \wh X\rnn_{s-})\,d(\bar
Y\rnn-\wh Y\rnn)_s+[\bar Y\rnn-\wh Y\rnn]_t-|\bar Y\rnn_t-\wh
Y\rnn_t|^2,
\end{align*}
which implies that
\begin{align}
\nonumber|\bar X\rnn_t-\wh X\rnn_t|^2
&\leq [\bar Y\rnn-\wh Y\rnn]_t+\frac{1}{\ro}\int_0^{t}
|\bar X\rnn_{s-}-\wh X\rnn_{s-}|\,d(|\bar K\rnn|+|\wh K\rnn|)_s\\
&\qquad+ 2\int_0^t(\bar X\rnn_{s-}- \wh X\rnn_{s-})\,d(\bar
Y\rnn-\wh Y\rnn)_s +R^{n,1}_t\nonumber\\
&= I^{n,1}_t+I^{n,2}_t+I^{n,3}_t+R^{n,1}_t. \label{eq4.12}
\end{align}
Since we may assume that $f$ is   Lipschitz continuous, we have
\begin{align*}
I^{n,1}_t&=\sum_{k;\tk\leq t}\big |\int_{\tkk}^{\tk}
(f(\bar X^n_s)-f(\wh X^n_s))ds\frac{\Delta Z_k}{\Delta t_k}\big |^2\\
&\leq C\sum_{k;\tk\leq t}|\bar X^n_{\tkk}-\wh X^n_{\tkk}|^2|\Delta Z_k|^2\\
&\quad+C\sup_{s\leq t}|(\bar X^n_s-\bar X\rnn_s)+
(\wh X^n_s-\wh X\rnn_s)|^2\sum_{k;\tk\leq t}|\Delta Z_k|^2\\
&=C\int_0^t|\bar X\rnn_{s-}-\wh X\rnn_{s-}|^2\,d[Z\rn]_s +R^{n,2}_t.
\end{align*}
Similarly, by using Lipschitz continuity of $f'f$ we get
\begin{align*}
|I^{n,3}_t|&=|\sum_{k;\tk\leq t}
(\bar X^n_{\tkk}-\wh X^n_{\tkk})
(\int_{\tkk}^{\tk}(f(\bar X^n_s)-f(\wh X^n_s))ds\frac{\Delta Z_k}{\Delta t_k})|\\
&\leq|\sum_{k;\tk\leq t}(\bar X^n_{\tkk}-\wh X^n_{\tkk})
(f(\bar X^n_{\tkk})-f(\wh X^n_{\tkk}))(\Delta Z_k)|\\
&\quad+|\sum_{k;\tk\leq t}(\bar X^n_{\tkk}-\wh X^n_{\tkk})
(\int_{\tkk}^{\tk}(\int_{\tkk}^{s}(f'(\bar X^n_u)d\bar X^n_u
-f'(\wh X^n_s)d\wh X^n_u)ds\frac{\Delta Z_k}{\Delta t_k})|\\
&\leq|\int_0^t(\bar X\rnn_{s-}-\wh X\rnn_{s-})
(f(\bar X\rnn_{s-})-f(\wh X\rnn_{s-}))\,dZ\rn_s|\\
&\quad+C\int_0^t|\bar X\rnn_{s-}-\wh X\rnn_{s-}|^2\,d[Z\rn]_s +R^{n,3}_t
\end{align*}
and
\begin{align*}
|R^{n,3}_t|&\leq C\sup_{s\leq t}|(\bar X^n_s-\bar X\rnn_s)+
(\wh X^n_s-\wh X\rnn_s)||\sum_{k;\tk\leq t}
|\bar X^n_{\tkk}-\wh X^n_{\tkk}||\Delta Z_k|^2\\
&\quad+|\sum_{k;\tk\leq t}(\bar X^n_{\tkk}-\wh X^n_{\tkk})
(\int_{\tkk}^{\tk}(\int_{\tkk}^{s}(f'(\bar X^n_u)d\bar K^n_u)\,ds
\frac{\Delta Z_k}{\Delta t_k}|\\
&\leq C\sup_{s\leq t}|(\bar X^n_s-\bar X\rnn_s)+
(\wh X^n_s-\wh X\rnn_s)|\int_0^t|\bar X^n_{s-}-\wh X^n_{s-}|\,d[Z\rn]_s\\
&\quad+C\max_{k;\tk\leq t}|\Delta Z_k|\int_0^t|\bar X\rnn_{s-}-\wh
X\rnn_{s-})|\,d|\bar K\rnn|_s.
\end{align*}
Substituting the last three estimates into (\ref{eq4.12}) we
obtain
\begin{align}
|\bar X\rnn_t-\wh X\rnn_t|^2&\leq C\int_0^t
|\bar X\rnn_{s-}-\wh X\rnn_{s-}|^2\,d[Z\rn]_s\label{eq4.13}\\
&\quad+\frac{1}{\ro}\int_0^{t}|\bar X\rnn_{s-}-\wh X\rnn_{s-}|\,
d(|\bar K\rnn|+|\wh K\rnn|)_s\nonumber\\
&\quad+|\int_0^t(\bar X\rnn_{s-}-\wh X\rnn_{s-})
(f(\bar X\rnn_{s-})-f(\wh X\rnn_{s-}))\,dZ\rn_s| +R^{n}_t,\nonumber
\end{align}
where $R^n_t=R^{n,1}_t+R^{n,2}_t+R^{n,3}_t$, $t\in\Rp$, $n\in\N$.
Simple calculation shows that for any $q\in\Rp$,
$\epsilon_n=E\sup_{t\leq q}|R^n_t|\to 0$. Fix $q\in\Rp$. Later on
we  will restrict our attention to processes which in addition are
are stopped at $q$. Set $A^n_t=[Z\rn]_t+|V\rn|_t+|\bar
K\rnn|_t/\ro+|\wh K\rnn|_t/\ro$. Clearly, there is $C_1>0$ such
that $A^n_{\infty}+[M\rn]_{\infty}+\langle
M\rn\rangle_{\infty}\leq C_1$.   On the other hand, by
(\ref{eq4.13}), there is $C_2>0$ such that for any $({\cal
F}\rn_t)$ stopping time $\gamma^n$,
\begin{align*}
&E\sup_{t<\gamma_n}|\bar X\rnn_t-\wh X\rnn_t|^2
\leq C_2\int_0^{\gamma^n-}\sup_{u\leq s}
|\bar X\rnn_{u-}-\wh X\rnn_{u-}|^2\,dA^n_s\\
&\quad\qquad+E(\sup_{t\leq \gamma^n-}
|\int_0^t(\bar X\rnn_{s-}-\wh X\rnn_{s-})
(f(\bar X\rnn_{s-})-f(\wh X\rnn_{s-}))\,dM\rn_s|) +\epsilon_n.
\end{align*}
Since $f$ is  Lipschitz continuous, it follows from the version of
Metivier-Pellaumail inequality proved in  Pratelli \cite{Pra} and
Schwartz inequality that
\begin{align*}
&E(\sup_{t< \gamma^n}|\int_0^t(\bar X\rnn_{s-}-\wh X\rnn_{s-})
(f(\bar X\rnn_{s-})-f(\wh X\rnn_{s-}))\,dM\rn_s|)\\
&\qquad\leq c
E(\int_0^{\gamma_n-}|\bar X\rnn_{s-}-\wh X\rnn_{s-})|^4\,
d([M\rn]+\langle M\rn\rangle)_s)^{1/2}\\
&\qquad\leq cE\sup_{t< \gamma^n}|\bar X\rnn_{s-}-\wh X\rnn_{s-})
|(\int_0^{\gamma_n-}|\bar X\rnn_{s-}-\wh X\rnn_{s-})|^2\,d([M\rn]
+\langle M\rn\rangle)_s)^{1/2}\\
&\qquad\leq c(E\sup_{s<\gamma^n}|\bar X\rnn_{s-}-\wh X\rnn_{s-})|^2)^{1/2}
(E\int_0^{\gamma_n-}|\bar X\rnn_{s-}-\wh X\rnn_{s-})|^2\,d([M\rn]
+\langle M\rn\rangle)_s)^{1/2}\\
&\qquad\leq \frac{1}{2}E\sup_{s<\gamma^n}|\bar X\rnn_{s-}-\wh X\rnn_{s-})|^2
+c'E\int_0^{\gamma_n-}|\bar X\rnn_{s-}-\wh X\rnn_{s-})|^2\,d([M\rn]
+\langle M\rn\rangle)_s.
\end{align*}
Therefore, for any $({\cal F}\rn_t)$ stopping time $\gamma^n$,
\begin{align*}
E\sup_{t<\gamma_n}&|\bar X\rnn_t-\wh X\rnn_t|^2\\
&\leq (2C_2+2c')E\int_0^{\gamma^n-}
\sup_{u\leq s}|\bar X\rnn_{u-}-\wh X\rnn_{u-}|^2\,
d(A^n+[M\rn]+\langle M\rn\rangle)_s+2\epsilon_n.
\end{align*}
From the stochastic version of Gronwall's lemma (see, e.g.,
\cite[Lemma 2]{Ma} or \cite[Lemma C1]{Sl3}) it now follows that
\[
E\sup_{t<q\wedge\tau^ n_b}|\bar X\rnn_t-\wh X\rnn_t|^2\leq
2\epsilon_n\exp\{(2C_2+2c')C_1)\}\lra0.
\]
Hence we conclude (\ref{eq4.10}) and completes the proof.
\end{proof}
\begin{remark} {\rm \label{rem4}
In general, if $Z$ is discontinuous, the solutions of
(\ref{eq2.6})  cannot be approximated by solutions of
(\ref{eq4.20}). To see this, let us consider $Z$ such that $Z_t=0$
if $t<1$  and $Z_t=1$, otherwise. Set $\tk=k/n$, $n\in\N$,
$k\in\N\cup\{0\}$. If $f$ is bounded and Lipschitz continuous then
there exists a unique solution $(\bar X^n,\bar K^n)$ of
(\ref{eq4.20}). Moreover, it is easy to check that $\bar X^n_t\to
X_0$  if $t<1$ and  $\bar X^n_t\to X_1$ if $t\geq1$, where $(X,K)$
is a solution of   the reflecting equation
\[X_t=X_0+\int_0^t f(X_s)ds\Delta Z_1+K_t,\quad t\in\Rp.\]
However, the limit of $\{\bar X^n\}$ need not be a solution of the
Skorokhod problem with jumps  (in general we do not have the
property $X_1=\Pi_{\bar D}(X_0+\int_0^1 f(X_s)ds\Delta Z_1)$).
Consequently, the limit need not be a solution of (\ref{eq2.6})
driven by $Z$.
 }
\end{remark}
\mbox{}\\
{\bf Acknowledgements}\\[1mm]The author thank the referee for  careful reading of the paper   and
many valuable remarks.
 Research supported by Polish NCN grant no.  2012/07/B/ST1/03508.

\begin{thebibliography}{aaa}
\bibitem{AS}
S. Aida, K. Sasaki, Wong-Zakai approximation of solutions to
reflecting  stochastic differential equations on domains in
Euclidean spaces,  {   Stochastic Process. Appl.} { 123} (2013)
3800--3827.
\bibitem{DP}
H. Doss, P. Priouret, Support d'un processus de reflexion,  Z.
Wahrsch. Verw. Gebiete 61 (3) (1982) 327--345.
\bibitem{EK}
S.N. Ethier, T.G. Kurtz,  { Markov Processes},  Wiley,  New York 1986.
\bibitem{ES}
L.C. Evans, D.W. Stroock, An approximation scheme for reflected
stochastic differential equations, Stochastic Process. Appl.  121
(2011) 1464--1491.
\bibitem{GK}
I. Gy{\"o}ngy, N. Krylov,  Existence of strong solutions for {I}t{\^o}
  stochastic equations via approximations, {Probab. Theory Related Fields}
  {105} (1996) 143--158.
 \bibitem{Jac}
J. Jacod,  Convergence en loi de semimartingales et variation quadratique,
{  Lecture Notes in Math}  { 721}
 Springer--Verlag 1979.
 \bibitem{Jak}
A. Jakubowski, { A non-Skorohod topology on the Skorohod space}, {  EJP}
{ 2} (1997) 1-21.
 \bibitem{JMP}
A. Jakubowski,  J. M\'emin, G. Pages, {Convergence en loi des suites
d'int\'egrales stochastiques sur l'espace $D^1$  de Skorokhod},
{   Probab. Theory Related Fields}  {81} (1989) 111--137.
\bibitem{KH}
A. Kohatsu-Higa, Stratonovich type SDEs with normal reflection driven by semimartingales, Sankhya 63 A (2) (2001), 194-228.
\bibitem{Kur}
T.G. Kurtz, { Random time changes and convergence in distribution under the
Meyer--Zheng conditions}, {   Ann. Probab.}
{ 19} (1991) 1010--1034.
\bibitem{KP}
T.G. Kurtz, P. Protter,
{Weak limit theorems for stochastic integrals and
stochastic differential equations},   {   Ann. Probab.}
{19} (1991), 1035--1070.
\bibitem{KP1}
T.G. Kurtz, P. Protter,
{ Wong--Zakai corrections,  random evolutions and
simulation schemes for SDE's},  Proc. Conference in Honor Moshe Zakai
65th Birthday,
Haifa,{  Stochastic Analysis} (1991), 331--346.
\bibitem{KPP}
T.G. Kurtz, E. Pardoux,  P. Protter,
{ Stratonovich stochastic differential
equations driven by general semimartingales},
   Ann. Inst. Henri Poincare  31 (2) (1995) 351--377.
\bibitem{LSz}
P.L. Lions, A.S. Sznitman, {Stochastic Differential Equations with
Reflecting Boundary Conditions},
{  Comm. Pure and Appl. Math.}  {XXXVII} (1983) 511--537.
\bibitem{Ma}
V. Mackievicius, { ${\cal S}^p$ stability of symmetric stochastic
differential equations with discontinuous driving semimartingales},
{Ann. Inst. Henri Poincar\'e}  { B 23} (1987) 575--592.
\bibitem{Ma1}
S. Marcus, {Modeling and analysis of stochastic differential equations
driven by point processes},  {   IEEE Transaction on Information
Theory}
{24} (1978), 164--172.
\bibitem{Mar}
S. Marcus, { Modeling and approximation of stochastic differential equations
driven by semimartingales},  {   Stochastics}
{4} (1981) 223--245.
\bibitem{MS}
J. M\'emin,  L. S\l omi\'nski, {Condition  UT  et stabilit\'e en loi des
solutions d'\'equations diff\'erentielles stochastiques},  {
S\'em. de Probababilite XXV, Lecture Notes in Math.}  { 1485}
 Springer--Verlag,
Berlin Heidelberg New York 1991 162--177.
\bibitem{MZ}
P.A. Meyer and  W.A. Zheng,
{Tightness criteria for laws of semimartingales},
{   Ann. Inst. Henri Poincar\'e}  {B 20}
(1984) 353--372.
\bibitem{Me1}
P.A. Meyer, { Un cours sur int\'egrales stochastiques}.
{  \em S\'em. de Probababilite X
Lecture Notes in Math.} 511  Springer--Verlag,  Berlin Heidelberg New York 1976.
 \bibitem{Pe}
 R. Petterson, Wong-Zakai approximations for reflecting  stochastic differential equations, Stoch. Anal. Appl. 17 (4) (1999) 609--617.
\bibitem{Pra}
M. Pratelli, Majoration dans $L^p$  du type Metivier-Pellaumail pour les semimartingales, Sem. de Probab. XVII Lect. Notes in Math.  986 Springer New York (1983) 125--131.
\bibitem{RX}
J. Ren, S. Xu, A transfer principle for multivalued stochastic differential equations, J. Funct. Anal.  256 (9) (2009) 2780--2814.
\bibitem{RX1}
J. Ren, S. Xu, Support theorem  for stochastic variational  inequalities,  Bull. Sci Math. 134 (8)  (2010) 826--856.
\bibitem{Sa}
Y. Saisho, { Stochastic differential equations for multi--dimensional
domain
with reflecting boundary},
{  Probab. Theory Related Fields}  { 74} (1987)
455--477.
\bibitem{Sl}
L. S\l omi\'nski, { Stability of strong solutions of stochastic differential
equations},  {   Stochastic Process. Appl.}
{ 31} (1989) 173--202.
\bibitem{Sl1}
L. S\l omi\'nski, {On existence,  uniqueness and stability of solutions
 of multidimensional  SDE's with reflecting boundary conditions},
{   Ann. Inst. H. Poincar\'e}
 { 29.2} (1993) 163--198.
\bibitem{Sl2}
L. S\l omi\'nski, {On approximation of solutions
 of multidimensional  SDE's with reflecting boundary conditions},
{   Stoch. Process. Appl.} { 50} (1994) 197--219.
\bibitem{Sl3}
L. S\l omi\'nski,   Stability of stochastic differential
equations driven by general semimartingales. { Diss. Math.} {
CCCXLIX} (1996) 1--113.
\bibitem{St}
C. Stricker, { Loi de semimartingales et crit\`eres de compacit\'e},
 {   S\'em. de Probab. XIX Lect. Notes in Math. }
  {1123}
Springer--Verlag,  Berlin Heidelberg New York 1985.
\bibitem{Ta}
H. Tanaka, { Stochastic differential equations with reflecting boundary
condition in convex regions},
 {   Hiroshima Math. J.}  { 9} (1979) 163--177.
\bibitem{WZ}
E. Wong, M. Zakai,
{On the convergence of ordinary integrals to stochastic
 integrals},
{ Ann. Math. Statist.}  { 36} (1965) 1560--1564.
\bibitem{WZ1}
E. Wong, M. Zakai,  On the relation between ordinary and stochastic differential equations,  Internat. J. Energ. Sci. 3 (1965) 213--229.
\bibitem{zh1}
T. Zhang, On the strong solutions of one-dimensional differential
equations with reflecting boundary, {   Stochastic Process. Appl.}
50 (1994) 135--147.
\bibitem{zh2}
T. Zhang, Strong Convergence of Wong-Zakai Approximations of
Reflected SDEs in a Multidimensional General Domain, Potential Anal. 41 (2014), 783-815.
\end{thebibliography}
\end{document}